\documentclass[10pt,twoside]{article}
\usepackage{amssymb,enumerate}
\usepackage{amsmath,amsthm}
\usepackage{graphicx}
\usepackage{xcolor}
\usepackage{amssymb}
\newtheorem{thm}{Theorem}

\newtheorem{remark}{Remark}
\newtheorem{proposition}{Proposition}

\begin{document}

\title{Dynamics of a discrete-time mixed oligopoly Cournot-type model with three time delays}

\author{Loredana Camelia Culda$^{1}$ \and Eva Kaslik$^{1,*}$ \and
Mihaela Neam\c{t}u$^{1}$
}

\date{\noindent $^1$ West University of Timi\c{s}oara, Bd.  V. P\^{a}rvan nr. 4, 300223, Timi\c{s}oara, Romania\\
$^*$ Corresponding Author: {\it eva.kaslik@e-uvt.ro}}

\maketitle

\begin{abstract}
The paper analyzes the interactions among one public firm and $n$ private firms on the market, in the framework of a discrete-time Cournot game with time delay. The production of the public firm is influenced by previous output levels of private firms. The productions of private companies are influenced by the past productions of the public company, as well as by the previous productions of the other private companies.
The associated nonlinear system admits two equilibrium points: the positive one and the boundary equilibrium. After the stability analysis, we obtained that the boundary equilibrium point is a saddle
point. If there is no delay, for the positive equilibrium point we have determined the stability region. Then, for different particular cases of delays, we found the conditions for which the positive equilibrium is asymptotically stable. The flip and Neimark-Sacker bifurcations are investigated. In addition, numerous numerical examples are performed to reveal the complex dynamic behavior of the system.
\end{abstract}

\section{Introduction}
\label{sec:1}

Game theory is the field of study that focuses on analyzing the interactions between multiple individuals or teams in a game, given certain conditions, in order to determine the optimal strategies for each party. In mathematical economics, oligopoly theory is a topic of interest and the earliest branch of chaotic dynamics, which is based on research on chaos theory and bifurcation theory using different dynamical systems, has found extensive applications. The literature has taken into account a variety of oligopoly models, including those with or without product differentiation and those with one or more products. Time delay models are used to reflect actual conditions when there are delays in the decision-making processes, lead time, information implementation, or execution time \cite{guerrini2018delay,matsumoto2015delay,matsumoto2015nonlinear}.

 In a discrete or continuous time setting, dynamic duopoly games have been examined in the context of quantity-setting firms. Players with homogenous or heterogeneous characteristics have the following alternatives for their strategies: naive, adaptive, or bounded rational \cite{matsumoto2015nonlinear, agiza2003nonlinear,bischi2000multistability,bischi2000global,cao2021stability,chen2016global,elsadany2014dynamic,gori2015continuous,howroyd1984cournot,pecora2018heterogenous, zhou2023complexity}. In the field of economics, expectations pertains to the predictions or perspectives that individuals in charge of decision-making have regarding forthcoming prices, sales, incomes, or other relevant factors. According to the adaptive expectations hypothesis, the current expectations are a combination of past expectations.


 The practical application of a game like this closely mirrors economic reality, and it is commonly used in oligopolies. In \cite{askar2018exploration}, which relates to dynamic Cournot oligopoly games, there are three concurrent firms with bounded rationality that are all based on the utility CES function. Recently, a Cournot-Theocharis oligopoly model with a single time delay has been investigated in \cite{canovas2023delayed}, considering that firms make decisions based on adaptive expectations and assuming that information on competitors is only available after a time lag. In the numerical analysis of the equilibria for the corresponding nonlinear discrete-time mathematical model, complex behavior is found.
The dynamics of a mixed triopoly game, in which a public firm competes against two private firms, are examined in \cite{wang2021complex}. The equilibrium points are identified and their local stability is examined, taking into account both quantity and price competition. 
 In \cite{andaluz2020dynamic} the stability of the Nash equilibrium is examined for a dynamic model with $n$ firms which compete in an isoelastic demand setting with non-unitary elasticity framework. Additionally, it is noted in \cite{haraguchi2016cournot}  that privately owned firms may run into financial difficulties, which could force those firms to be nationalized, while state-controlled public firms are important to their market competitors. In \cite{culda2022dynamic}, the interactions of one public firm and $n$ private firms on the market are considered and the analysis of the corresponding  discrete-time Cournot game with two time delays is discussed.

 Also, in \cite{cavalli2023endogenous}, it has been observed that players tend to choose strategies that differ from those found in the Tullock Nash equilibrium. In \cite{purificato2023debt}, the authors studied the interactions between fiscal and monetary authorities in a monetary union during a debt stabilization process, assuming that policy authorities do not coordinate and cannot perfectly predict each other's decisions. 

The present study aims to advance earlier research by studying the impact of the number of private firms in the market on the stability of the equilibrium when information delays are taken into account. This is motivated by mixed competition as well as the multiple delays in the decision-making process. 
To be more explicit, we take into account one public firm and $n$ private firms that are engaged in the production of differentiated products within the context of a dynamic oligopoly game. The public firm decides on its output based on the expected marginal payoff, or the social surplus, while taking into account the historical production levels of the private firms \cite{wang2021complex, kawasaki2020endogenous}. Utilizing reaction functions and previous output from the public firm, the outputs of the private firms are determined. 

The major result is the characterization of the stability of the Nash equilibrium with respect to the quantity of private firms, the level of product differentiation, the adjustment parameter, and three time delays.

This paper's structure is outlined in the following. In Section \ref{sec:2}, the mathematical model is given, and two equilibrium points are identified: the positive equilibrium and the boundary equilibrium.
The local stability analysis for the boundary and the positive equilibrium points is covered in Section \ref{sec:3}. The theoretical results are exemplified using numerical simulations in Section \ref{sec:4}, which is then followed by conclusions and a discussion of future research options.

\section{Mathematical model}
\label{sec:2}
Let $q_0$ represent the output of the public firm and $q_i$, $i=\overline{1,n}$, stand for the output of the private firm $i$. The retail price for the public firm is $p_0$ and $p_i$, $i=\overline{1,n}$, for the private firm $i$.

The aim of the representative consumer is to maximize the following function \cite{singh1984price}:
\begin{equation}\label{beneficiu}
    U(q_0,q_1,...,q_n)-\sum_{i=0}^n p_iq_i,
\end{equation} 
 where \cite{wang2021complex, singh1984price, askar2021nonlinear}:
 
$$U(q_0,q_1,...,q_n)=a\sum_{i=0}^n q_i-\dfrac{b}{2}\left(\sum_{i=0}^nq_i^2+\delta \sum \limits_{i=0}^n\sum \limits_{i\not= j} q_iq_j\right),$$ with $a, b$ real positive numbers and $\delta\in(0,1)$ the degree of product differentiation.

The maximization problem of \eqref{beneficiu} leads to: 
\begin{equation}\label{pi}
p_i=a-bq_i-b\delta\sum \limits_{j=0, j\not = i}^nq_j, \quad i=\overline{0,n}.\end{equation}

The objective of the public firm is to maximize the social surplus, while the purpose of the private firm is to maximize the profit function. The profit function of firm $i$ is given by:
\begin{equation}\label{profit}
P_i=(p_i-c_i)q_i, \quad i=\overline{0,n},
\end{equation} with $c_i$ the marginal cost of firm $i$, and the social surplus is \cite{wang2021complex}:
\begin{equation}\label{sw}
 SW(q_0,q_1,..,q_n)=a\sum_{i=0}^nq_i-\dfrac{b}{2}\left(\sum_{i=0}^nq_i^2+\delta \sum \limits_{i=0}^n\sum \limits_{i\not= j} q_iq_j\right)-\sum_{i=0}^np_iq_i+\sum_{i=0}^nP_i.
\end{equation} 
We consider the same marginal costs for all the private firms: $c_1=c_2=...=c_n=c$.

The maximization problems lead to:
\begin{equation*}
\begin{cases}
    \dfrac{\partial SW}{\partial q_0}(q_0,q_1,...,q_n)=0,\\
    \dfrac{\partial P_i}{\partial q_i}(q_0,q_1,...,q_n)=0, i=\overline{1,n},
\end{cases}
\end{equation*} or equivalently 
\begin{equation}\label{q0}
a-c_0-bq_0-b\delta\sum \limits_{i=1}^nq_i=0.
\end{equation} 
with $a>c_0\geq c$, and 
\begin{equation}\label{qi0}
p_i-c-bq_i=0, \quad i=\overline{1,n}.
\end{equation} 
From (\ref{pi}), (\ref{q0}) and (\ref{qi0}) we have:
\begin{equation}\label{qi}
\begin{cases}
q_0=\dfrac{a_0}{b}-\delta\sum \limits_{i=1}^nq_i,\\
q_i=\dfrac{a_1}{2b}-\dfrac{\delta}{2}\sum \limits_{j=0,j\not = i}^nq_j, \quad i=\overline{1,n},
\end{cases}
\end{equation} where $a_0=a-c_0>0$ and $a_1=a-c>0$. 

As the public firm has bounded rationality and the private firm $i$, $i=\overline{1,n}$, is naive, the dynamical equations for the outputs are given by \cite{culda2022dynamic}:
$$q_0(t+1)=q_0(t)+\alpha q_0(t)\left[a_0-bq_0(t)-b\delta \sum \limits_{i=1}^n q_i(t)\right],$$ where $\alpha$ is the positive adjustment parameter,
$$q_j(t+1)=\dfrac{a_1}{2b}-\dfrac{\delta}{2}\sum_{i=0,i\neq j}^n q_i(t),\, j=\overline{1,n}.$$

As in \cite{elsadany2010dynamics}, we consider that the output of the public firm is influenced by the past output levels of the private firms (at time $t-\tau_1$, $\tau_1>0$). Moreover, as in \cite{culda2022dynamic}, the productions of private firms are set up according to the past productions (at time $t-\tau_0$, $\tau_0>0$) of the public firm. Furthermore, in the present paper we also adjust the productions of private firms with the past
productions (at time $t-\tau_2$, $\tau_2>0$) of the other private firms. 

Therefore, in this paper, we investigate the following nonlinear discrete-time  mathematical model with time delays: 
\begin{equation}\label{system}
\begin{cases}
q_0(t+1)=q_0(t)+\alpha q_0(t)\left[a_0-bq_0(t)-b\delta \sum \limits_{i=1}^n q_i(t-\tau_1)\right]\\
q_j(t+1)=\dfrac{a_1}{2b}-\dfrac{\delta}{2}q_0(t-\tau_0)-\dfrac{\delta}{2}\sum\limits_{i=1,i\neq j}^n q_i(t-\tau_2)\quad ,\quad j=\overline{1,n}.
\end{cases}
\end{equation}

The equilibrium points of the discrete dynamical system \eqref{system} are:
\[
E_0=(0, q^\star,q^\star,...,q^\star),~\text{where }q^\star =\dfrac{a_1}{b[2+(n-1)\delta]}
\]
and

\[E_+=(q_0^\star,q_1^\star,q_1^\star...,q_1^\star),~\text{where } q_0^\star=\dfrac{[2+(n-1)\delta]a_0- n \delta a_1}{b[2+(n-1)\delta-n\delta^2]}~,~
q_1^\star=\dfrac{a_1-\delta a_0}{b[2+(n-1)\delta-n\delta^2]}.\]
Due to the fact that $\delta\in(0,1)$, the positivity of the equilibrium $E_+$ is equivalent to the following assumptions: 
\begin{align*}
\text{(A.1)}&\qquad [2+(n-1)\delta]a_0>n\delta a_1~,\\
\text{(A.2)}&\qquad a_1>\delta a_0~.
\end{align*}

\section{Local stability and bifurcation analysis}
\label{sec:3}

The liniarized system at one of the equilibrium points $E=(q_0^e,q_1^e,q_1^e,...,q_1^e)\in\{E_0,E_+\}$ is of the form: 
\begin{equation}\label{eq.liniarized}
    y(t+1)=A^ey(t)-B_0 y(t-\tau_0)-B_1^ey(t-\tau_1)-B_2 y(t-\tau_2)
\end{equation}
where 
\[y(t)=\left[
\begin{array}{cccc}
    q_0(t)-q_0^e & q_1(t)-q_1^e & \dots & q_n(t)-q_1^e 
\end{array}
\right]^T\]
and the matrices $A^e,B_0,B_1^e,B_2$ are given below: 

\[A^e=\begin{bmatrix}
1+\alpha(a_0-2bq_0^e-nb\delta q_1^e) & 0 & \ldots& 0\\
0 & 0 & \ldots& 0\\
\vdots & \vdots & \ddots&\vdots\\
0& 0 & \ldots &0
\end{bmatrix}\qquad,\qquad B_0=\begin{bmatrix}
0 & 0 & \ldots & 0\\
\delta/2 & 0 &  \ldots& 0\\
\vdots & \vdots & \ddots&\vdots\\
\delta/2 & 0 &  \ldots& 0\\
\end{bmatrix}\]
\[
B_1^e=\begin{bmatrix}
0 & b\alpha\delta q_0^e & \ldots & b\alpha\delta q_0^e\\
0 & 0 &  \ldots& 0\\
\vdots & \vdots & \ddots&\vdots\\
0 & 0 &  \ldots& 0\\
\end{bmatrix}
\qquad,\qquad B_2=\begin{bmatrix}
 0 & -\delta/2 & \ldots& -\delta/2\\
 -\delta/2 &0 & \ldots&-\delta/2\\
 \vdots & \vdots & \ddots&\vdots\\
-\delta/2&-\delta/2&\ldots&0
\end{bmatrix}\]

The characteristic equation of system \eqref{eq.liniarized} can be obtained using the $\mathcal{Z}$-transform method, and is given as follows:
\[
    \det\left(A^e-B_0\lambda^{-\tau_0}-B_1^e\lambda^{-\tau_1}-B_2\lambda^{-\tau_2}-\lambda I\right)=0,
\]
or equivalently: 
\begin{equation}\label{ec.char.general}
\left(\!\lambda\!-\!\dfrac{\delta}{2}\lambda^{-\tau_2}\right)^{\!n-1}\!\!\left[nb\alpha q_0^e\dfrac{\delta^2}{2}\lambda^{-\tau_0-\tau_1}\!\!-\!\!\left(\lambda\!\!-\!\!1\!-\!\alpha(a_0\!-\!2bq_0^e\!-\!nb\delta q_1^e)\right)\left(\!\lambda\!+\!(n\!-\!1)\dfrac{\delta}{2}\lambda^{-\tau_2}\right)\!\right]\!=\!0.
\end{equation}

In what follows, we analyze each of the equilibrium points $E_0$ and $E_+$. 

\subsection[The boundary equilibrium E0]{The boundary equilibrium $E_0$}

\begin{thm}
If assumption $(A.1)$ holds, the boundary equilibrium point $E_0$ is a saddle
point.
\end{thm}
\begin{proof}
At the boundary equilibrium point $E_0$, as $q_0^e=0$ and $q_1^e=q^\star$, the characteristic equation \eqref{ec.char.general} reduces to: 
\begin{equation}\label{ec.char.gen.}
\left(\lambda-\dfrac{\delta}{2}\lambda^{-\tau_2}\right)^{n-1}\left(\lambda\!-\!1-\alpha \frac{(2+(n-1)\delta)a_0-n\delta a_1}{2+(n-1)\delta}\right)\left(\lambda+(n\!-\!1)\dfrac{\delta}{2}\lambda^{-\tau_2}\right)=0.
\end{equation}
We notice that one root of \eqref{ec.char.gen.} is 
$\lambda_1=1+\dfrac{(2+(n-1)\delta)a_0-n\delta a_1}{2+(n-1)\delta}>1$, due to assumption $(A.1)$.

On the other hand, we can notice that the characteristic equation \eqref{ec.char.gen.} also admits some roots inside the unit disk, which satisfy:
\[\lambda^{\tau_2+1}=\frac{\delta}{2}<1.\]
In conclusion, the equilibrium $E_0$ is a
saddle point of system \eqref{system}.
\end{proof}
\subsection[The positive equilibrium Eplus]{The positive equilibrium $E_+$}

As in this case $q_0^e=q_0^\star$ and $q_1^e=q_1^\star$, 
the characteristic equation \eqref{ec.char.general} becomes
\begin{equation}\label{ec.char.Eplus}
\left(\lambda-\dfrac{\delta}{2}\lambda^{-\tau_2}\right)^{n-1}\left[\varepsilon_0(\varepsilon_1+1)\lambda^{-\tau_0-\tau_1}-(\lambda+\varepsilon_1)\left(\lambda+\varepsilon_2\lambda^{-\tau_2}\right)\right]=0,
\end{equation}
where 
\begin{equation}\label{eq.beta}
\varepsilon_0=n\dfrac{\delta^2}{2}>0\quad,\quad\varepsilon_1+1=\alpha\dfrac{[2+(n-1)\delta]a_0- n \delta a_1}{2+(n-1)\delta-n\delta^2}>0\quad\text{ and }\quad \varepsilon_2=(n-1)\dfrac{\delta}{2}>0.\end{equation}

Some of the roots of \eqref{ec.char.Eplus} are given by
\[\lambda^{\tau_2+1}=\frac{\delta}{2}<1,\]
and hence, these roots belong to the open unit disk. Therefore, the stability of the equilibrium point $E_+$ is determined by the roots of the following reduced equation:
\begin{equation}\label{ec.char.reduced}
\varepsilon_0(\varepsilon_1+1)\lambda^{-\tau_0-\tau_1}-(\lambda+\varepsilon_1)\left(\lambda+\varepsilon_2\lambda^{-\tau_2}\right)=0.
\end{equation}

In the absence of time delays, based on the Schur-Cohn stability conditions, the following result has been obtained in  \cite{culda2022dynamic} regarding the asymptotic stability of the equilibrium point $E_+$:

\begin{thm}
\label{thm.stab.no.delays}
If assumptions $(A.1)$ and $(A.2)$ hold, when $\tau_0=\tau_1=\tau_2=0$, the equilibrium point $E_+$ is asymptotically stable if and only if $(\varepsilon_0,\varepsilon_1,\varepsilon_2)$ belong to the \textbf{delay-free stability region} defined by the following inequalities: 
\begin{equation}\label{ineq.stab.dom}
   \varepsilon_2<1\qquad\text{and}\qquad \varepsilon_1<\dfrac{1-\varepsilon_2-\varepsilon_0}{1-\varepsilon_2+\varepsilon_0}.
\end{equation}
\end{thm}
This theorem represents a generalization of the result  presented in \cite{wang2021complex}, where the case of two private firms ($n=2$) has been investigated.

\begin{remark}\label{rem.ineg}
    If inequalities \eqref{ineq.stab.dom} hold, it follows that $\varepsilon_1<1$ and: 
\begin{align*}
(1-\varepsilon_1)\left(1-\varepsilon_2\right)&>\left(1-\dfrac{1-\varepsilon_2-\varepsilon_0}{1-\varepsilon_2+\varepsilon_0}\right)\left(1-\varepsilon_2\right)\\
&=\varepsilon_0\dfrac{2(1-\varepsilon_2)}{1-\varepsilon_2+\varepsilon_0}\\
&>\varepsilon_0(\varepsilon_1+1).
\end{align*}
    
\end{remark}

In what follows, we describe two situations related to the time delays, where inequalities \eqref{ineq.stab.dom} provide sufficient conditions for the asymptotic stability of $E_+$. 

\begin{thm}\label{thm.stab.delays.tau2.0}
Assume that $\tau_0\geq 0$, $\tau_1\geq 0$ and $\tau_2=0$. If the assumptions $(A.1)$ and $(A.2)$ hold and inequalities \eqref{ineq.stab.dom} are satisfied, the equilibrium point $E_+$ is asymptotically stable.
\end{thm}

\begin{proof}
Theorem \ref{thm.stab.no.delays} provides that if assumptions $(A.1)$, $(A.2)$ and inequalities \eqref{ineq.stab.dom} hold, the equilibrium $E_+$ is asymptotically stable for null time delays. Assuming by contradiction that asymptotic stability of the equilibrium point is lost for certain values of the time delays, based on the continuous dependence of the roots of the characteristic equation \eqref{ec.char.reduced} on $\tau_0,\tau_1$,  it follows that there exist critical values $(\tau_0^*,\tau_1^*)$, such that the equation \eqref{ec.char.reduced} has some roots $\lambda$ belonging to the unit circle.

The characteristic equation \eqref{ec.char.reduced} can be rewritten as:
\[
\varepsilon_0(\varepsilon_1+1)\lambda^{-\tau_0-\tau_1}=(\lambda+\varepsilon_1)\left(\lambda+\varepsilon_2\right).
\]
Assuming that $\lambda=e^{i\theta}$, with $\theta\in[0,\pi]$, satisfies the above equation for $(\tau_0,\tau_1)=(\tau_0^*,\tau_1^*)$, taking the absolute value of both sides of the equation leads to: 
\[\left|e^{i\theta}+\varepsilon_1\right|\left|e^{i\theta}+\varepsilon_2\right|=\varepsilon_0(\varepsilon_1+1),\]
or equivalently:
\begin{equation}\label{eq.mu}
\left[2\varepsilon_1 \cos\theta+\varepsilon_1^2+1\right]\left[2\varepsilon_2\cos\theta+\varepsilon_2^2+1\right]=\varepsilon_0^2(\varepsilon_1+1)^2.
\end{equation}

Taking into account inequalities \eqref{ineq.stab.dom}, we deduce that $\varepsilon_1<1$. 

On the one hand, if $\varepsilon_1>0$, based on Remark \ref{rem.ineg}, the following inequalities hold for the left hand side of equation \eqref{eq.mu}:
\begin{align*}
\left[2\varepsilon_1\cos\theta+\varepsilon_1^2+1\right]\left[2\varepsilon_1\cos\theta+\varepsilon_2^2+1\right]&\geq \left[-2\varepsilon_1+\varepsilon_1^2+1\right]\left[-2\varepsilon_2+\varepsilon_2^2+1\right]
\\& =(1-\varepsilon_1)^2\left(1-\varepsilon_2\right)^2
\\&>\varepsilon_0^2(\varepsilon_1+1)^2~,
\end{align*}
which contradicts equality \eqref{eq.mu}. 

On the other hand, if $\varepsilon_1\leq 0$, denoting by $P(\cos\theta)$ the left hand side of equation \eqref{eq.mu}, it follows that $P$ is a concave quadratic polynomial, and hence, using similar arguments as in the previous computations, we obtain:
\[P(\cos\theta)\geq \min\{P(-1),P(1)\}>\varepsilon_0^2(\varepsilon_1+1)^2,\]
which again, contradicts equality \eqref{eq.mu}.

Consequently, if the assumptions of the theorem hold, the equilibrium point $E_+$ is asymptotically stable for any time delays $\tau_0$ and $\tau_1$.
\end{proof}

\begin{thm}\label{thm.stab.equal.delays}
Assume that $\tau_0+\tau_1=\tau_2$. If the assumptions $(A.1)$ and $(A.2)$ hold and inequalities \eqref{ineq.stab.dom} are satisfied, the equilibrium point $E_+$ is asymptotically stable.
\end{thm}

\begin{proof}
If $\tau_0+\tau_1=\tau_2:=\tau$, the characteristic equation may be written as:
\[
\varepsilon_0(\varepsilon_1+1)=(\lambda+\varepsilon_1)\left(\lambda^{\tau+1}+\varepsilon_2\right).\]
Let us assume that there exists $\tau\geq 0$ such that this characteristic equation has a root $\lambda=e^{i\theta}$, with $\theta\in(0,\pi)$.

Let us consider $\rho_1>0$, $\rho_2>0$ and $\phi_1,\phi_2\in(0,2\pi)$ such that:
\[
\begin{cases}
\lambda+\varepsilon_1=\rho_1 e^{i\phi_1}\\
\lambda^{\tau+1}+\varepsilon_2=\rho_2 e^{i\phi_2}
\end{cases}\]
Hence, the characteristic equation now implies:
\[
\varepsilon_0(\varepsilon_1+1)=\rho_1\rho_2e^{i(\phi_1+\phi_2)}.
\]
and therefore:
\[
\begin{cases}
\rho_1\rho_2=\varepsilon_0(\varepsilon_1+1)\\
\phi_1+\phi_2=2\pi
\end{cases}\]
 we obtain:
\[
\begin{cases}
\cos\theta+\varepsilon_1=\rho_1\cos\phi_1\\
    \sin\theta=\rho_1\sin\phi_1\\
    \cos(\tau+1)\theta+\varepsilon_2=\rho_2\cos\phi_2\\
\sin(\tau+1)\theta=\rho_2\sin\phi_2
\end{cases}
\]
From the second equation, we deduce $\phi_1\in(0,\pi)$, and hence, $\phi_2=2\pi-\phi_1$. Moreover, eliminating $\theta$ from the previous system, we get:
\[\begin{cases}
\rho_1\rho_2=\varepsilon_0(\varepsilon_1+1)\\
    \rho_1^2-2\varepsilon_1\rho_1\cos\phi_1+\varepsilon_1^2-1=0\\
     \rho_2^2-2\varepsilon_2\rho_2\cos\phi_1+\varepsilon_2^2-1=0
\end{cases}
\]
Solving the last two quadratic equations and keeping in mind that $\rho_1>0$ and $\rho_2>0$ we have: 
\[\rho_k=\varepsilon_k\cos\phi_1+\sqrt{\varepsilon_k^2\cos^2\phi_1+1-\varepsilon_k^2},\quad\text{for }k\in\{1,2\}.\]
Denoting $\mu=\cos\phi_1\in[-1,1]$ and replacing in the first equation of the above system, we obtain: 
\[h(\mu):=\left[ \varepsilon_1\mu+\sqrt{\varepsilon_1^2\mu^2+1-\varepsilon_1^2} \right]\cdot\left[ \varepsilon_2\mu+\sqrt{\varepsilon_2^2\mu^2+1-\varepsilon_2^2} \right]=\varepsilon_0(\varepsilon_1+1).\]
It is easy to check that the function $h$ is monotonous (strictly increasing if $\varepsilon_1+\varepsilon_2>0$ and strictly decreasing otherwise), and hence:
\begin{align*}
h(\mu)&\geq \min\{h(-1),h(1)\}=\min\{(1-\varepsilon_1)(1-\varepsilon_2), (1+\varepsilon_1)(1+\varepsilon_2)\}\\
&=
\min\left\{(1-\varepsilon_1)\left(1-\varepsilon_2\right),(1+\varepsilon_1) \left(1+\varepsilon_2\right)\right\}>\varepsilon_0(\varepsilon_1+1),
\end{align*}
where Remark \ref{rem.ineg} has been employed. 
Therefore, we have arrived at a contradiction, and hence, if inequalities \eqref{ineq.stab.dom} hold, the positive equilibrium is asymptotically stable. However, for certain values of the time delays, the exact stability region may be larger than the delay-independent stability region \eqref{t0+t1=impar}. 
\end{proof}

\begin{remark}If either $\tau_2=0$ or $\tau_0+\tau_1=\tau_2$, Theorems 
 \ref{thm.stab.delays.tau2.0} and \ref{thm.stab.equal.delays} reveal that time delays may have a stabilizing effect on $E_+$. In these two cases, the delay-free stability regions provided by inequalities \eqref{ineq.stab.dom} are in fact, delay-independent stability regions of $E_+$. These regions have been exemplified in Figure \ref{fig.stab.regions}, for $a_0=2$ and $a_1=2.5$. We observe that for a larger number $n$ of private firms, smaller values of $\delta$ are needed for the stability of $E_+$, while slightly larger values of $\alpha$ are permissible.  
\end{remark}

\begin{figure}[htbp]
    \centering    \includegraphics[width=0.53\linewidth]{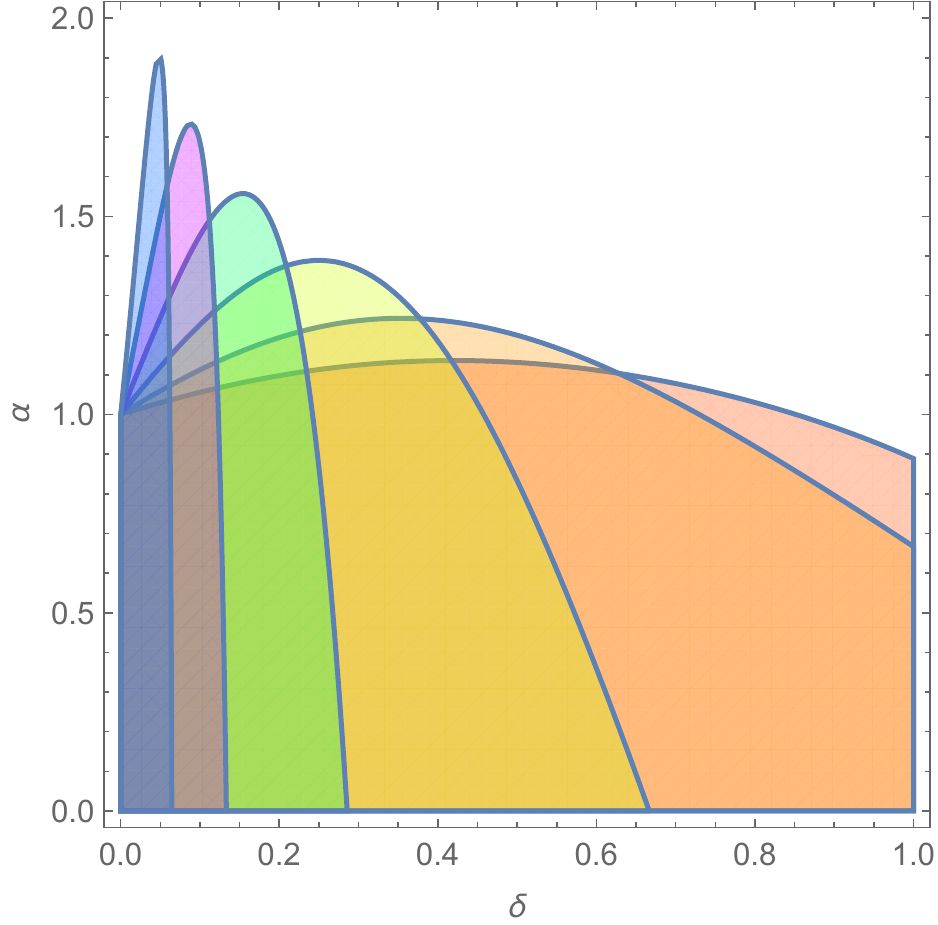}
    \caption{Stability regions (independent of time delays $(\tau_0,\tau_1,\tau_2)$ such that either $\tau_2=0$ or $\tau_0+\tau_1=\tau_2$) of the positive equilibrium point $E_+$ of system \eqref{system}
    in the $(\delta,\alpha)$ parameter plane, for different value of $n=2^k, k=\{0, 1, 2, 3, 4, 5\}$ (colored orange to blue). Here, $a_0=2$ and $a_1=2.5$.}
\label{fig.stab.regions}
\end{figure}
\begin{proposition}
A flip bifurcation takes place in a neighborhood of the equilibrium point $E_+$ if and only if:

\begin{equation}\label{even.odd}
    \varepsilon_1=\dfrac{1-\varepsilon_2(-1)^{\tau_2}-\varepsilon_0(-1)^{\tau_0+\tau_1}}{1-\varepsilon_2(-1)^{\tau_2}+\varepsilon_0(-1)^{\tau_0+\tau_1}}.
\end{equation}
\end{proposition}
\begin{remark}\label{rem.case}
We distinguish four cases presented below:
\begin{enumerate}
    \item[(i)] If $\tau_0+\tau_1$ and  $\tau_2$ are \textbf{even}, a flip bifurcation takes place in a neighborhood of the equilibrium point $E_+$ exactly at the boundary of the delay-free stability region given by inequalities \eqref{ineq.stab.dom}, i.e. when  \[\varepsilon_1=\dfrac{1-\varepsilon_2-\varepsilon_0}{1-\varepsilon_2+\varepsilon_0}~.\]
     \item[(ii)] If $\tau_0+\tau_1$ and  $\tau_2$ are \textbf{odd}, a flip bifurcation takes place in a neighborhood of the equilibrium point $E_+$ if and only if  \[\varepsilon_1=\dfrac{1+\varepsilon_2+\varepsilon_0}{1+\varepsilon_2-\varepsilon_0}~.\]
\item[(iii)] If $\tau_0+\tau_1$ is \textbf{even} and $\tau_2$ is \textbf{odd} a flip bifurcation takes place in a neighborhood of the equilibrium point $E_+$ if and only if  \[\varepsilon_1=\dfrac{1+\varepsilon_2-\varepsilon_0}{1+\varepsilon_2+\varepsilon_0}~.\]
\item[(iv)]  If $\tau_0+\tau_1$ is \textbf{odd} and $\tau_2$ is \textbf{even}, a flip bifurcation takes place in a neighborhood of the equilibrium point $E_+$ if and only if  \[\varepsilon_1=\dfrac{1-\varepsilon_2+\varepsilon_0}{1-\varepsilon_2-\varepsilon_0}~.\]
\end{enumerate}
\end{remark}
   
\begin{remark}
 To study the Neimark-Sacker bifurcation from the equation \eqref{ec.char.reduced} for $\lambda=e^{i\theta}$ and $\tau=\tau_0+\tau_1$, we have that \begin{equation}\label{ec.epsilon1.NS}
   \varepsilon_1=\dfrac{e^{i\theta} H(\theta, \tau, \tau_2, \varepsilon_2)-\varepsilon_0}{\varepsilon_0 -H(\theta, \tau, \tau_2, \varepsilon_2)}~,  
 \end{equation}where $H(\theta, \tau, \tau_2, \varepsilon_2)=e^{i\tau\theta}(e^{i\theta}+\varepsilon_2e^{-i\tau_2\theta})$ and if we take the real and imaginary part we obtain the following system:   
\[\begin{cases}
\Im(H)=\sin(\tau+1)\theta+\varepsilon_2\sin(\tau-\tau_2)\theta\\
\Re(H)=\cos(\tau+1)\theta+\varepsilon_2\cos(\tau-\tau_2)\theta
\end{cases}
\]
As $\varepsilon_1\in\mathbb{R}$, taking the imaginary part in equation \eqref{ec.epsilon1.NS} leads to:
\begin{equation}\label{ec.imaginar.NS}
    \varepsilon_0\cos\left(\tau+\frac{3}{2}\right)\theta-\varepsilon_0\varepsilon_2\cos\left(\tau-\tau_2+\frac{1}{2}\right)\theta=\cos\frac{\theta}{2}\left[1+\varepsilon_2^2+2\varepsilon_2\cos(\tau_2+1)\theta\right]
\end{equation}
and taking the real part in equation \eqref{ec.epsilon1.NS} gives:
$$\varepsilon_1=\dfrac{2\cos\frac{\theta}{2}[\varepsilon_0\cos\left(\tau+\frac{3}{2}\right)\theta+\varepsilon_0\varepsilon_2\cos\left(\tau-\tau_2+\frac{1}{2}\right)\theta]-\cos\theta (1+\varepsilon_2^2+2\varepsilon_2\cos(\tau_2+1)\theta)-\varepsilon_2^2}{\varepsilon_0^2+\varepsilon_2^2+2\varepsilon_2\cos(\tau_2+1)\theta-2\varepsilon_0\cos(\tau+1)\theta-2\varepsilon_0\varepsilon_2\cos(\tau-\tau_2)\theta+1}$$

In particular case, $\tau_2=\tau$ we obtain:
$$\varepsilon_1=\dfrac{2\cos\frac{\theta}{2}[\varepsilon_0\cos\left(\tau+\frac{3}{2}\right)\theta+\varepsilon_0\varepsilon_2\cos\frac{1}{2}\theta]-\cos\theta (1+\varepsilon_2^2+2\varepsilon_2\cos(\tau+1)\theta)-\varepsilon_2^2}{\varepsilon_0^2+\varepsilon_2^2+2\varepsilon_2\cos(\tau+1)\theta-2\varepsilon_0\cos(\tau+1)\theta-2\varepsilon_0\varepsilon_2+1}$$
and if $\tau_2=0$ it is obtain the equation:
$$\varepsilon_1=\dfrac{2\cos\frac{\theta}{2}[\varepsilon_0\cos\left(\tau+\frac{3}{2}\right)\theta+\varepsilon_0\varepsilon_2\cos\left(\tau+\frac{1}{2}\right)\theta]-\cos\theta (1+\varepsilon_2^2+2\varepsilon_2\cos(\tau+1)\theta)-\varepsilon_2^2}{\varepsilon_0^2+\varepsilon_2^2+2\varepsilon_2\cos\theta-2\varepsilon_0\cos(\tau+1)\theta-2\varepsilon_0\varepsilon_2\cos\tau\theta+1}$$
 \end{remark} 


\begin{figure}[htbp]
    \centering
    \includegraphics[width=1\linewidth]{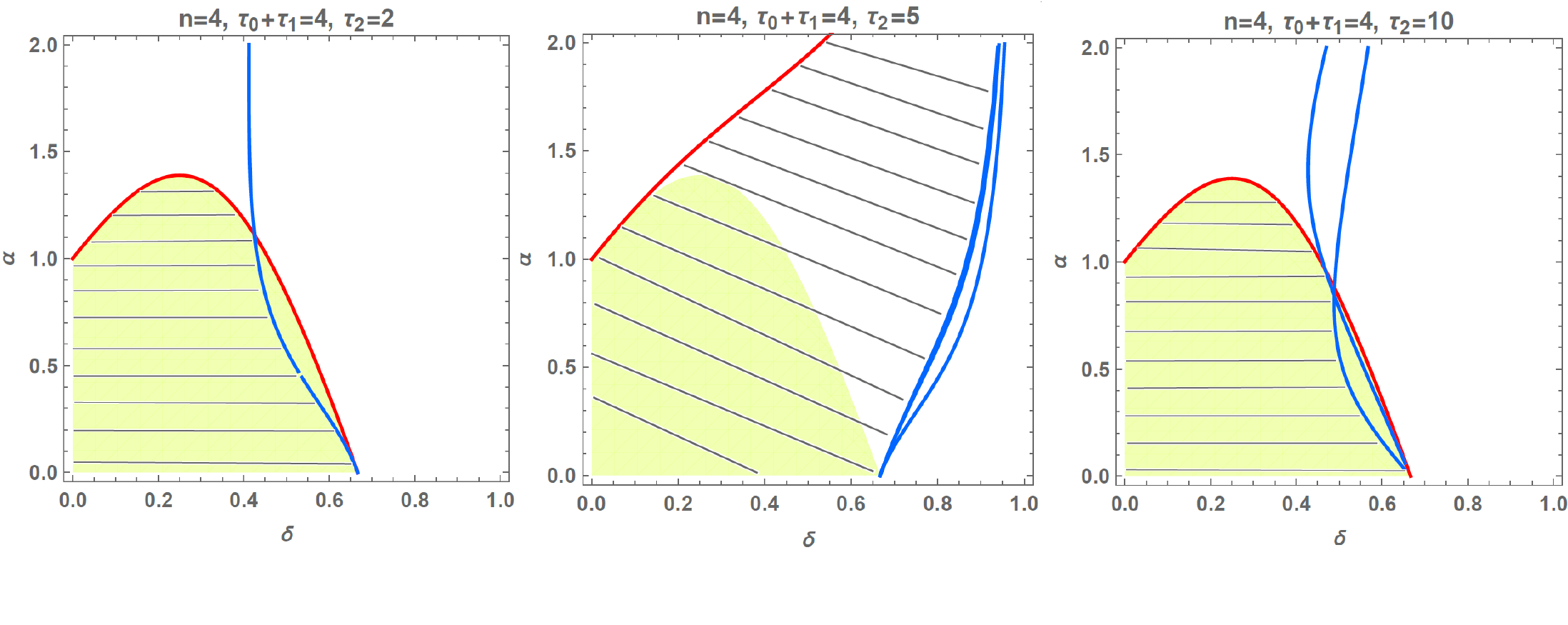}
    \caption{Stability region of the positive equilibrium point $E_+$ of system \eqref{system} for flip and Neimark-Sacker bifurcation with $n=4$ private firms and one public firm, with respect to $\alpha$. Time delays: $\tau_0+\tau_1$ - even and $\tau_2$ with different values.}
    \label{t0+t1=par}
\end{figure}
\begin{figure}[htbp]
    \centering
    \includegraphics[width=1\linewidth]{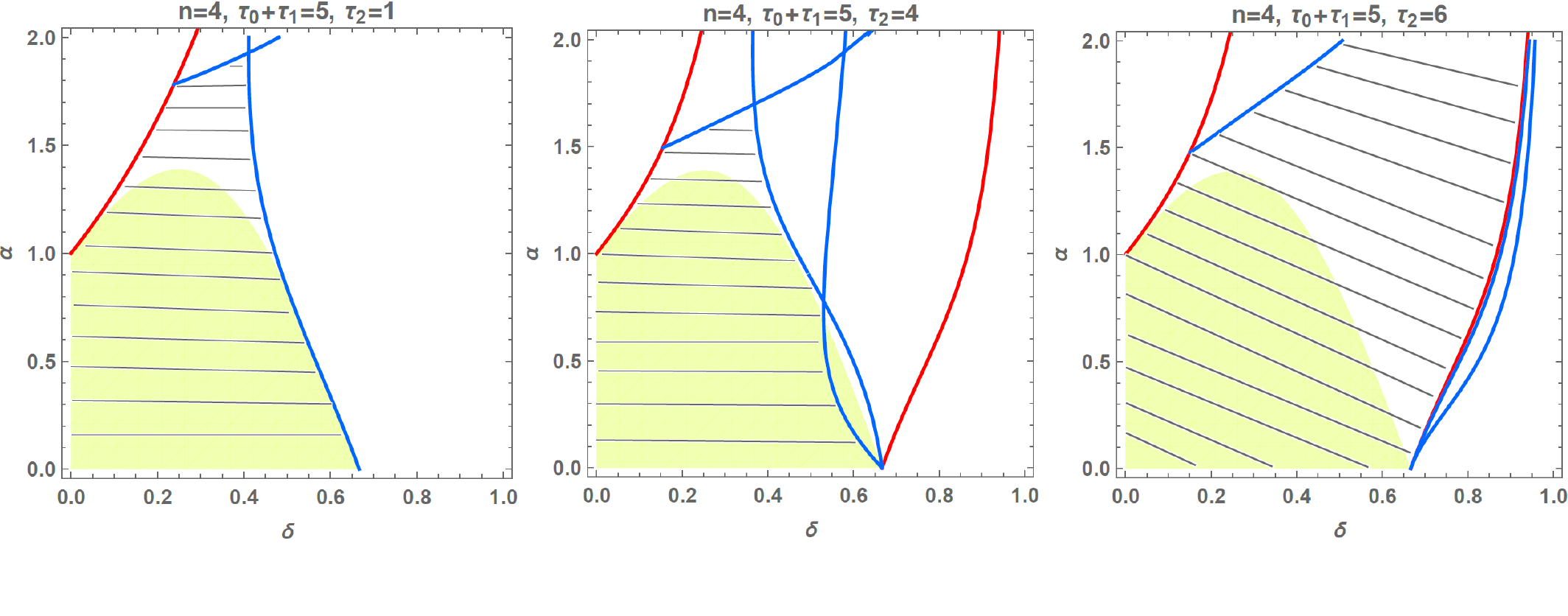}
    \caption{Stability region of the positive equilibrium point $E_+$ of system \eqref{system} for flip and Neimark-Sacker bifurcation with $n=4$ private firms and one public firm, with respect to $\alpha$. Time delays: $\tau_0+\tau_1=5$ and $\tau_2$ with different values.}
    \label{t0+t1=impar}
\end{figure}

\begin{figure}[htbp]
    \centering
    \includegraphics[width=1\linewidth]{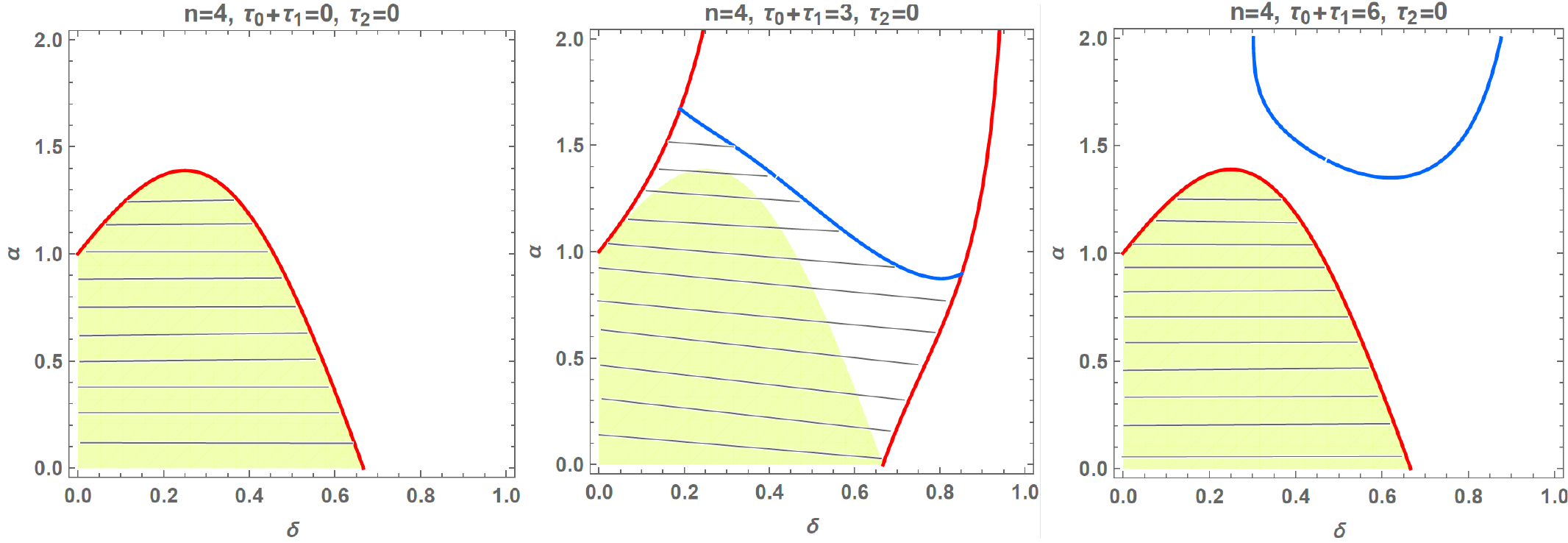}
    \caption{Stability region of the positive equilibrium point $E_+$ of system \eqref{system} for flip and Neimark-Sacker bifurcation with $n=4$ private firms and one public firm, with respect to $\alpha$. Time delays: $\tau_0+\tau_1$ with different values and $\tau_2=0$.}
    \label{t2=0}
\end{figure}
\begin{figure}[htbp]
    \centering
    \includegraphics[width=1\linewidth]{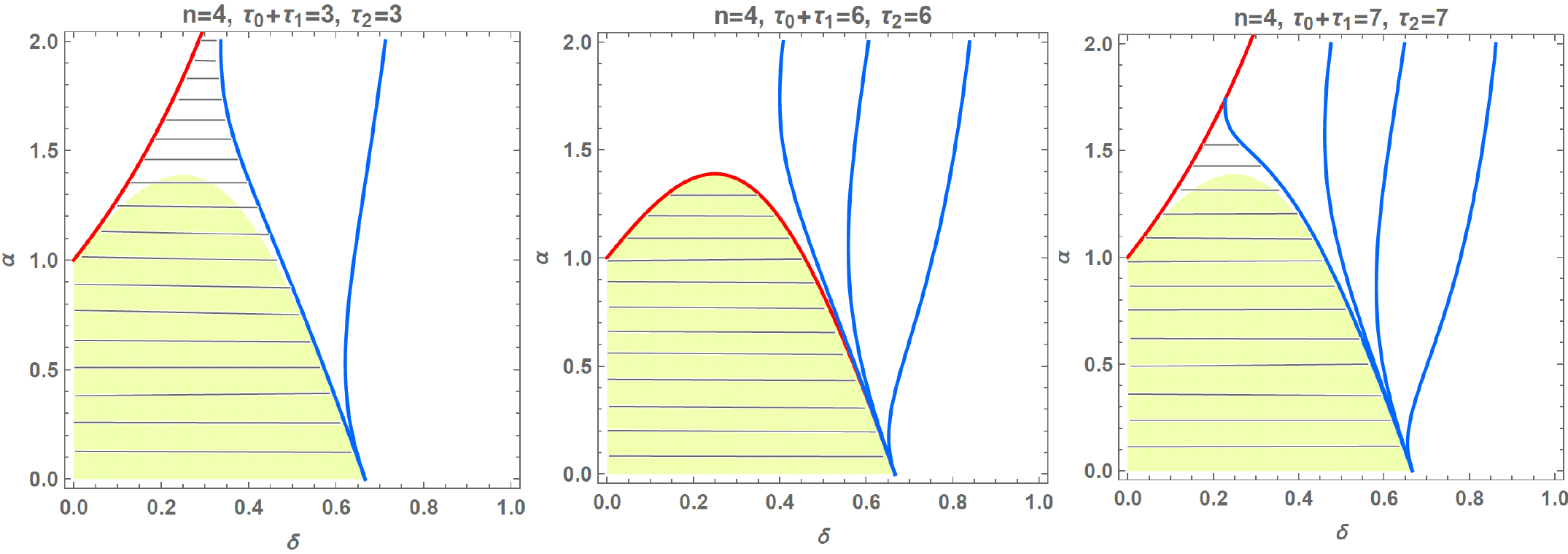}
    \caption{Stability region of the positive equilibrium point $E_+$ of system \eqref{system} for flip and Neimark-Sacker bifurcation with $n=4$ private firms and one public firm, with respect to $\alpha$. Time delays: $\tau_0+\tau_1=\tau_2$.}
    \label{tauegale}
\end{figure}

\subsection{Stability analysis in the absence of the public firm}
In the absence of public firm, $q_0(t)=0$, the nonlinear discrete-time mathematical model with delay \eqref{system} is reduced to: \begin{equation}\label{system.without.public.firm}
    q_j(t+1)=\dfrac{a_1}{2b}-\dfrac{\delta}{2}\sum\limits_{i=1,i\neq j}^n q_i(t-\tau_2)\quad ,\quad j=\overline{1,n}.
\end{equation} 
where the equilibrium points are $E_0^r=(q_0^\star, q_0^\star, ..., q_0^\star)$ and $E_+^r=(q_1^\star, q_1^\star, ..., q_1^\star)$.

The characteristic equation is given as follows: 
\begin{equation}\label{ec.char.without.public.firm}
\left(\!\lambda\!-\!\dfrac{\delta}{2}\lambda^{-\tau_2}\right)^{\!n-1}\!\!\left(\!\lambda\!+\!(n\!-\!1)\dfrac{\delta}{2}\lambda^{-\tau_2}\right)\!=\!0.\end{equation}

\section{Numerical examples}
\label{sec:4}

 To showcase our theoretical results, we examine a scenario with $4$ private firms and $1$ public firm, with the following fixed parameters: $a_0=2$, $a_1=2.5$, $b=1$, and $\delta=0.4$. Under these parameter values, the positive equilibrium point is calculated to be: \[E_+=(0.9375,0.664,0.664,0.664,0.664).\]

 By referencing inequalities \eqref{ineq.stab.dom}, we conclude that the positive equilibrium $E_+$ is asymptotically stable, regardless of the chosen values of the time delays $\tau_0$, $\tau_1$, and $\tau_2$, provided that $\alpha<\alpha^\star=1.185$. From the observations in Remark \ref{rem.case}, it can be deduced that in the event where $\tau_0 + \tau_1$ and $\tau_2$ are both even, a flip bifurcation will occur in a vicinity of the positive equilibrium at the critical value of the parameter $\alpha$, denoted by $\alpha^\star$. This is consistent with the bifurcation diagrams shown in Figures \ref{2210} and \ref{248} displayed with respect to the parameter $\alpha$ for the special cases $\tau_0=\tau_1=2,\tau_2=10$ and $\tau_0=2, \tau_1=4,\tau_2=8$ respectively.
In the first case, the flip bifurcation is followed by a period-doubling bifurcation at approximately $\alpha = 1.55$. Conversely, in Figure \ref{248}, the flip bifurcation is followed by a period-doubling bifurcation at around $\alpha = 1.38$ and a Neimark-Sacker bifurcation of the period-4 point at approximately $\alpha = 1.51$.

\begin{figure}[htbp]
    \centering
    \includegraphics[width=0.75\linewidth]{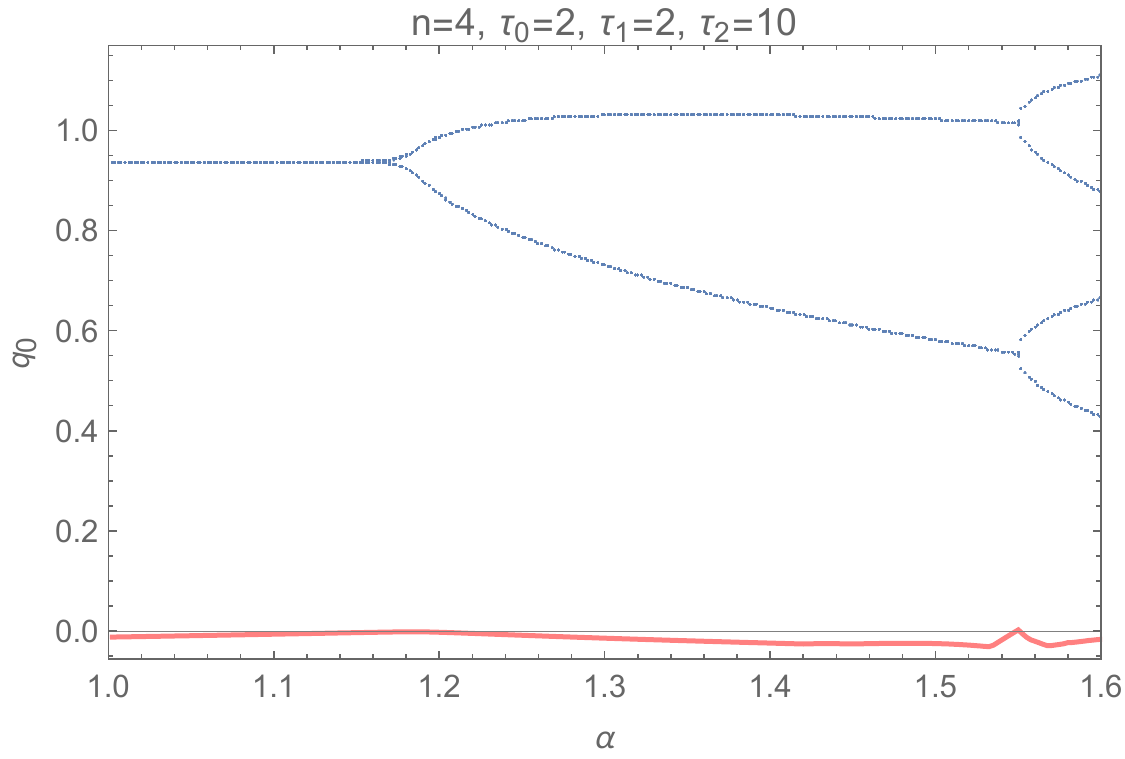} 
  \caption{Bifurcation diagram and largest Lyapunov exponent (shown in red) for system \eqref{system} with $n=4$ private firms and one public firm, with respect to $\alpha$. Fixed parameter values: $a_0=2$, $a_1=2.5$, $b=1$ and $\delta=0.4$. Time delays: $\tau_0=2$, $\tau_1=2$, $\tau_2=10$.}
  \label{2210}
\end{figure}

\begin{figure}[htbp]
    \centering
    \includegraphics[width=0.75\linewidth]{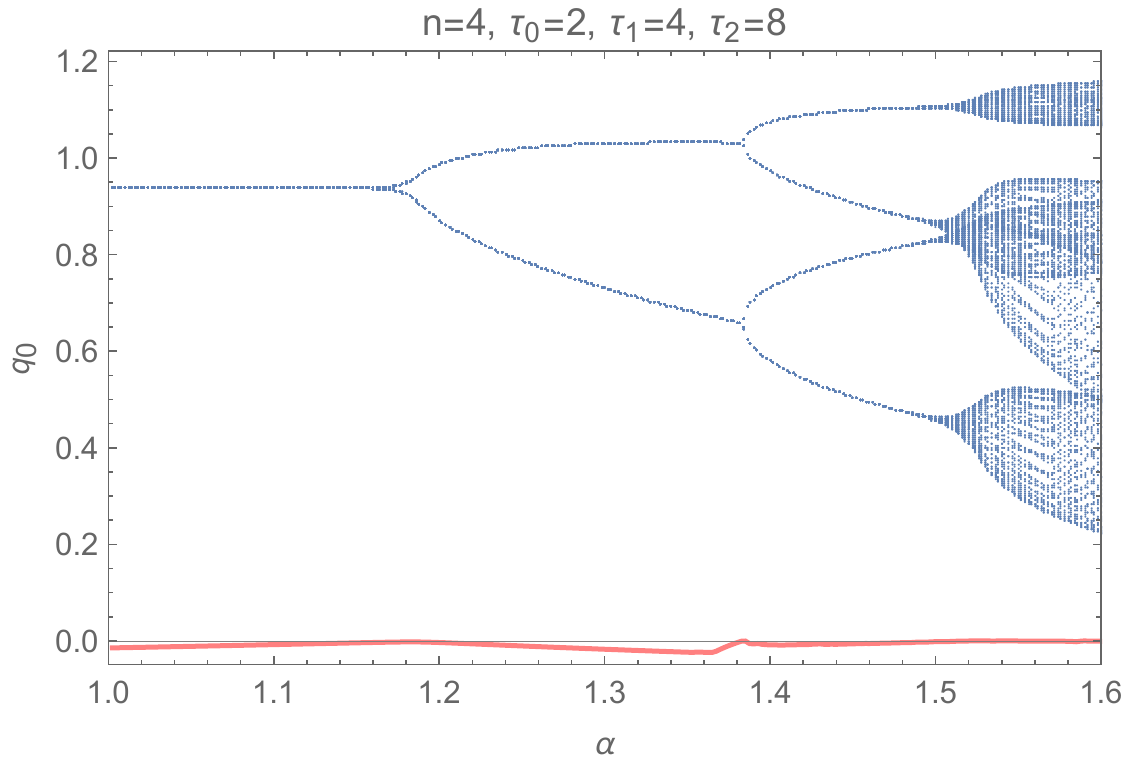}
  \caption{Bifurcation diagram and largest Lyapunov exponent (shown in red) for system \eqref{system} with $n=4$ private firms and one public firm, with respect to $\alpha$. Fixed parameter values: $a_0=2$, $a_1=2.5$, $b=1$ and $\delta=0.4$. Time delays: $\tau_0=2$, $\tau_1=4$, $\tau_2=8$.}
  
  \label{248}
\end{figure}

\begin{figure}[htbp]
    \centering
\includegraphics[width=1\linewidth]{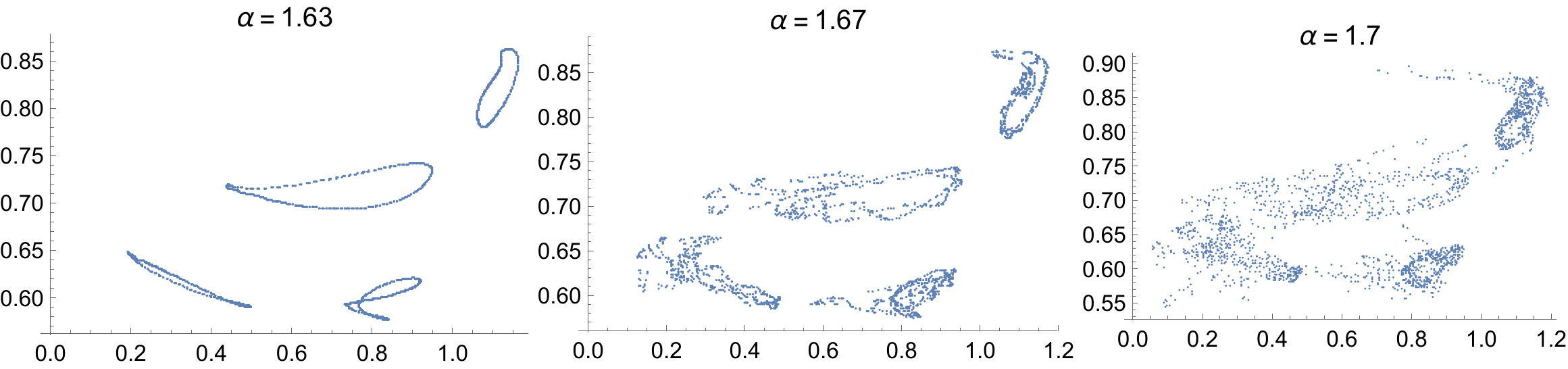}
  \caption{Phase portraits 
  for system \eqref{system} with $n=4$ private firms and one public firm, for various values of $\alpha$. Fixed parameter values: $a_0=2$, $a_1=2.5$, $b=1$ and $\delta=0.4$. Time delays: $\tau_0=2, \tau_1=4$ and $ \tau_2=8$.}
    \label{248faza}
\end{figure}

In contrast with the previous two examples, the bifurcation diagrams for the cases $\tau_0=5$, $\tau_1=3$, $\tau_2=3$ and $\tau_0=3$, $\tau_1=5$, $\tau_2=5$, displayed in Figures \ref{figbif4533} and \ref{figbif4355}, show that in these cases, the stability of the positive equilibrium $E_+$ is lost due to a Neimark-Sacker bifurcation. When $\tau_0=5$, $\tau_1=3$, $\tau_2=3$ (see Figure \ref{figbif4533}), a Neimark-Sacker bifurcation takes place at $\alpha\simeq 1.43$, and a stable limit cycle is formed. However, for $\alpha>1.6$, we observe the occurrence of a chaotic attractor. On the other hand, when $\tau_0=3$, $\tau_1=5$, $\tau_2=5$ (see Figure \ref{figbif4355}), we can only observe a Neimark-Sacker bifurcation that occurs for $\alpha\simeq 1.28$ and the resulting stable limit cycles persist for $\alpha>1.28$  (the largest Lyapunov exponent remains constantly null).

\begin{figure}[htbp]
    \centering
    \includegraphics[width=0.75\linewidth]{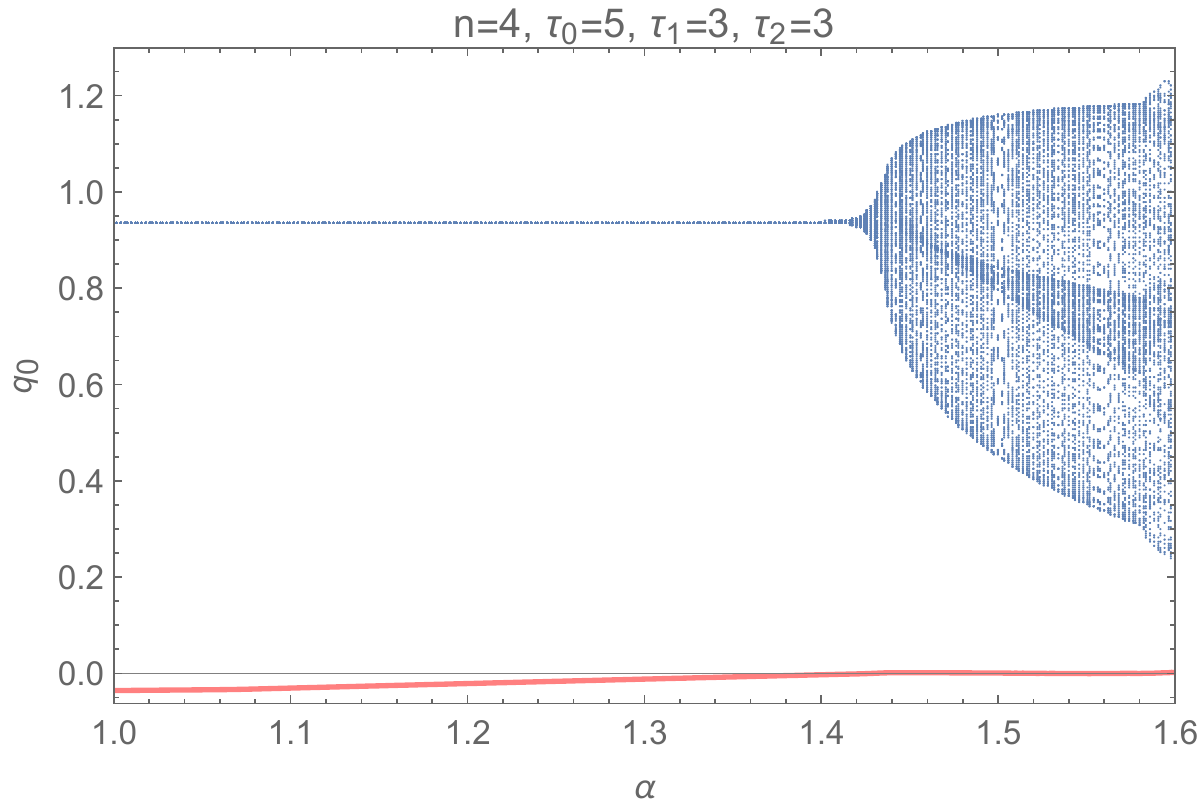}
  \caption{Bifurcation diagram and largest Lyapunov exponent (shown in red) for system \eqref{system} with $n=4$ private firms and one public firm, with respect to $\alpha$. Fixed parameter values: $a_0=2$, $a_1=2.5$, $b=1$ and $\delta=0.4$. Time delays: $\tau_0=5$, $\tau_1=3$, $\tau_2=3$.}
  \label{figbif4533}
\end{figure}

\begin{figure}[htbp]
    \centering
\includegraphics[width=1\linewidth]{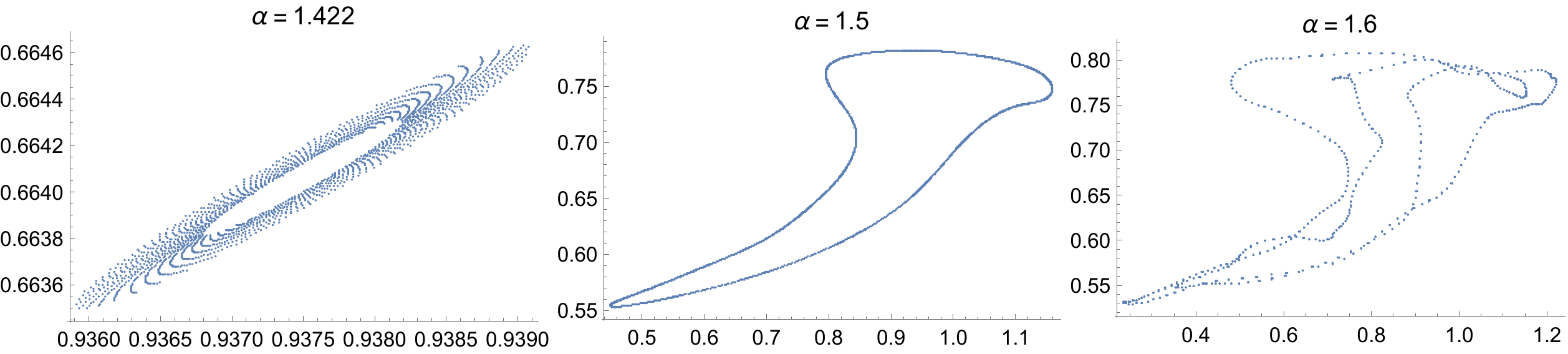}
  \caption{Phase portraits 
  for system \eqref{system} with $n=4$ private firms and one public firm, for various values of $\alpha$. Fixed parameter values: $a_0=2$, $a_1=2.5$, $b=1$ and $\delta=0.4$. Time delays: $\tau_0=5, \tau_1=3$ and $ \tau_2=3$.}
      \label{533faza}
\end{figure}

\begin{figure}[htbp]
    \centering
    \includegraphics[width=0.75\linewidth]{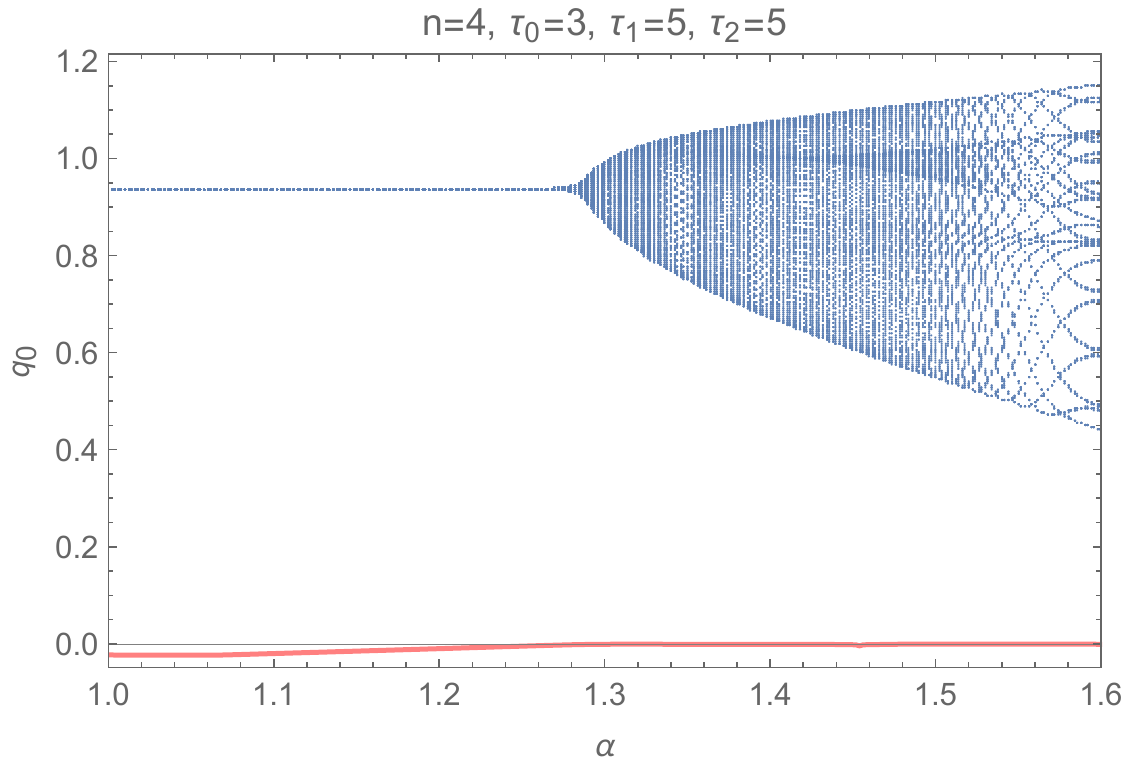}
  \caption{Bifurcation diagram and largest Lyapunov exponent (shown in red) for system \eqref{system} with $n=4$ private firms and one public firm, with respect to $\alpha$. Fixed parameter values: $a_0=2$, $a_1=2.5$, $b=1$ and $\delta=0.4$. Time delays: $\tau_0=3$, $\tau_1=5$, $\tau_2=5$.}
\label{figbif4355}  
\end{figure}

\begin{figure}[htbp]
    \centering
\includegraphics[width=1\linewidth]{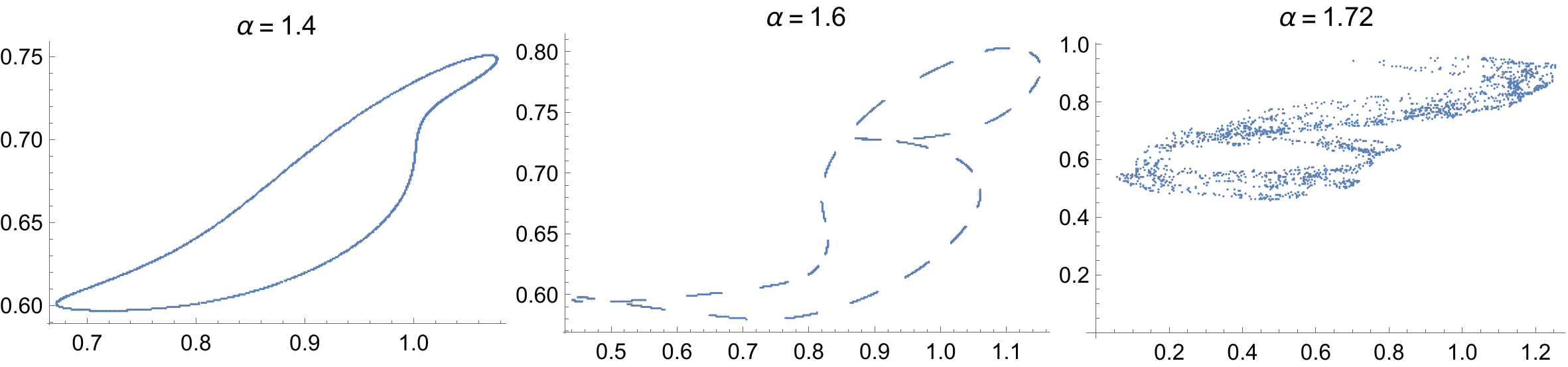}
  \caption{Phase portraits 
  for system \eqref{system} with $n=4$ private firms and one public firm, for various values of $\alpha$. Fixed parameter values: $a_0=2$, $a_1=2.5$, $b=1$ and $\delta=0.4$. Time delays: $\tau_0=3, \tau_1=5$ and $ \tau_2=5$.}
  
    \label{355faza}
\end{figure}

As a final example, as indicated by Remark \ref{rem.case}, if $\tau_0+\tau_1$ is even and $\tau_2$ is odd, the positive equilibrium $E_+$ loses its stability at $\alpha=\alpha^\star$. This is demonstrated in the bifurcation diagram from Figure \ref{975}  for the case of $\tau_0=9$, $\tau_1=7$, and $\tau_2=5$.  
The flip bifurcation at $\alpha\simeq 1.49$  is followed by a Neimark-Sacker bifurcation of the period-2 point for $\alpha\simeq 1.65$. Again, for sufficiently large values of the parameter $\alpha$, chaos arises, emphasized by the positive values of the largest Lyapunov exponent.

The phase portraits displayed in Figures \ref{248faza}, \ref{533faza}, \ref{355faza} and \ref{975faza} are consistent with the bifurcation diagrams, which illustrate the various dynamic regimes ranging from period doubling for small $\alpha$ to the appearance of chaos when $\alpha$ is sufficiently large.

\begin{figure}[htbp]
    \centering
\includegraphics[width=0.75\linewidth]{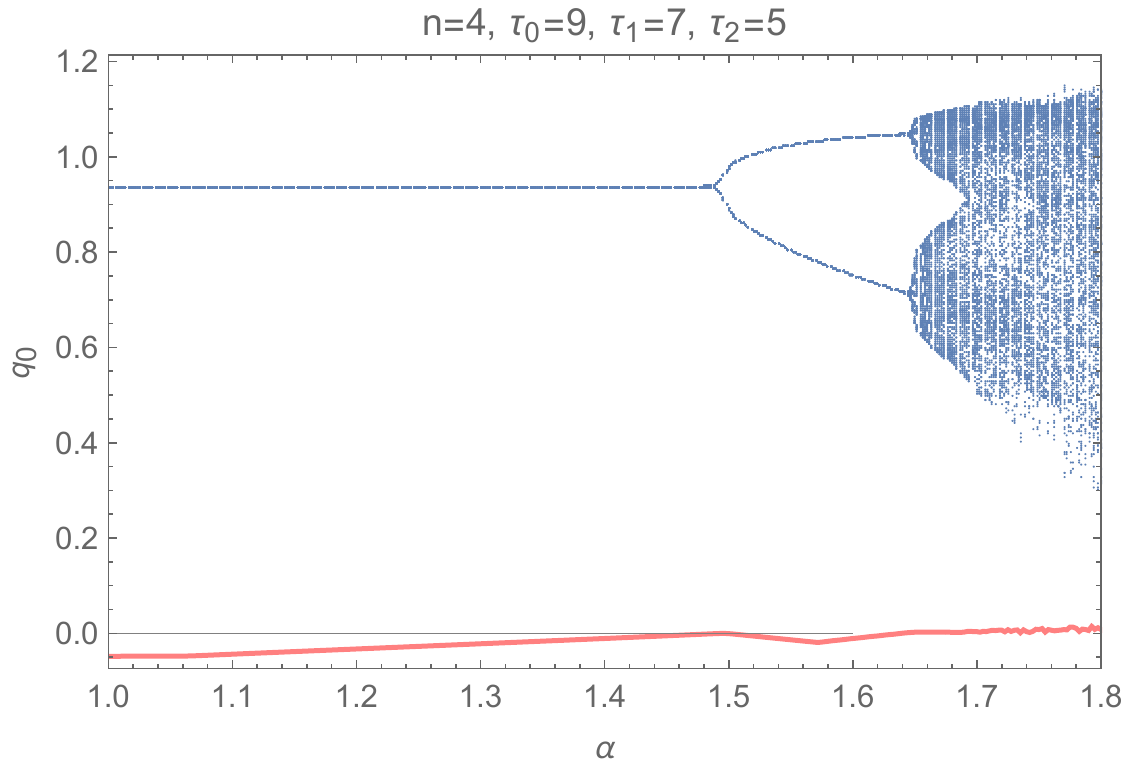}
  \caption{Bifurcation diagram and largest Lyapunov exponent (shown in red) for system \eqref{system} with $n=4$ private firms and one public firm, with respect to $\alpha$. Fixed parameter values: $a_0=2$, $a_1=2.5$, $b=1$ and $\delta=0.4$. Time delays: $\tau_0=9$, $\tau_1=7$, $\tau_2=5$.}
  \label{975}
\end{figure}

\begin{figure}[htbp]
    \centering
\includegraphics[width=1\linewidth]{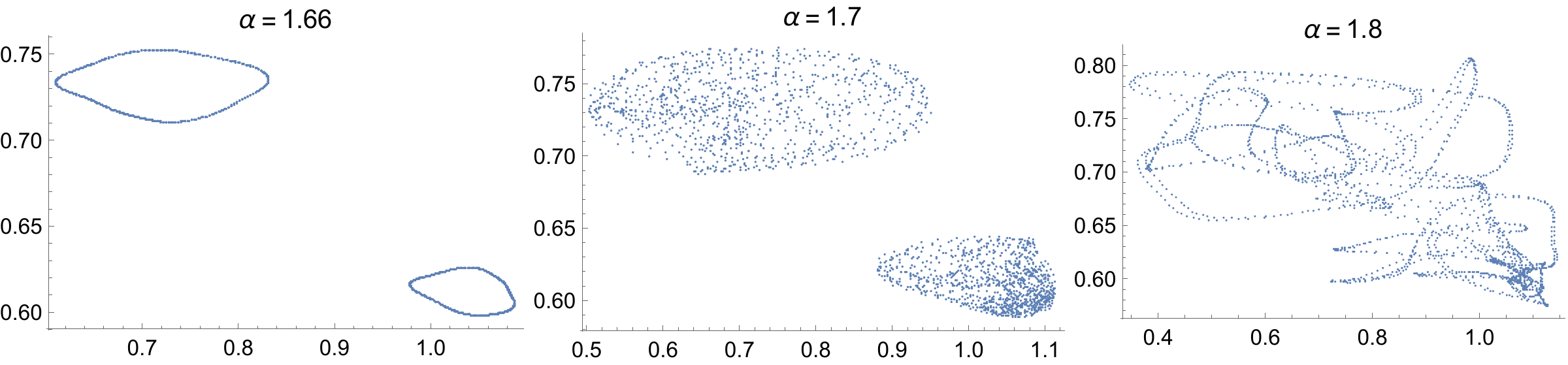}
  \caption{Phase portraits 
  for system \eqref{system} with $n=4$ private firms and one public firm, for various values of $\alpha$. Fixed parameter values: $a_0=2$, $a_1=2.5$, $b=1$ and $\delta=0.4$. Time delays: $\tau_0=9,\tau_1=7 $ and $\tau_2=5$.}
    \label{975faza}
\end{figure}

\section{Conclusions}
\label{sec:5}

The dynamics of an oligopoly game with product differentiation, in which $n$ private firms and a state-owned public firm coexist, have been examined in the current work. Two equilibrium points for the associated discrete-time mathematical model with three time delays have been established, and the local stability has been investigated.
The positive equilibrium $E_+$ is asymptotically stable when there is no delay and certain conditions are fulfilled. Additionally, we have identified the necessary conditions that ensure $E_+$ is asymptotically stable, irrespective of time delays.
We have demonstrated that in certain cases, the positive equilibrium point $E_+$ may be stabilized by the time delays. We have seen that as there are more private firms, the stability of the positive equilibrium requires smaller values of the degree of product differentiation, which is connected with slightly greater values of the adjustment parameter. When the time delays are large enough, numerical simulations show complicated dynamic behavior as well as the presence of chaos. Our findings emphasize the impact of different sets of time delays on the system's dynamics.

Our results generalize several findings from \cite{wang2021complex}, and they can be extended in the following ways: obtaining a thorough understanding of the Neimark-Sacker bifurcations occurring in the neighborhood of $E_+$; comprehending the potential paths leading to chaotic behavior in terms of the quantity of private firms and the time delays; and analyzing a mathematical model resembling this one in which the network of $n$ private firms does not have all-to-all connection.

 \bibliography{biblio.bib}

\end{document}